\documentclass[12pt]{article}

\usepackage[lmargin = 2cm, rmargin = 2cm, bmargin = 2cm,tmargin=2cm]{geometry}

\usepackage{titlesec}

\usepackage{microtype}
  
\usepackage{enumitem}

\usepackage{xfrac}

\usepackage{subfig}

\usepackage{titlesec}

\usepackage{graphicx}

\usepackage{amssymb}

\usepackage{amsmath}

\usepackage{cancel}

\usepackage{amsthm}

\usepackage{accents}

\usepackage{wrapfig}

\usepackage{mathrsfs}

\usepackage{eufrak}

\usepackage{tikz-cd}

\usepackage{xcolor}

\usepackage{hyphenat}

\usepackage{fancyhdr}

\usepackage{url}

\usepackage{mathdots}

\usepackage{etoolbox}

\usepackage{tgtermes}

\definecolor{ggreen}{rgb}{0,0.75,0.08}

\theoremstyle{plain}

\theoremstyle{definition}

\newtheorem{theorem}{Theorem}[section]

\newtheorem{lemma}[theorem]{Lemma}

\newtheorem{proposition}[theorem]{Proposition}

\newtheorem{corollary}[theorem]{Corollary}

\newtheorem{definition}[theorem]{Definition}

\newtheorem{question*}{Question}

\newtheorem{example}[theorem]{Example}

\newcommand{\A}{\ensuremath{\alpha}}

\newcommand{\K}{\ensuremath{\kappa}}
\newcommand{\B}{\ensuremath{\beta}}

\newcommand{\E}{\ensuremath{\varepsilon}}
\newcommand{\W}{\ensuremath{\omega}}

\newcommand{\G}{\ensuremath{\gamma}}

\newcommand{\GG}{\ensuremath{\Gamma}}
\newcommand{\RR}{\ensuremath{\mathbb R}}

\newcommand{\0}{\ensuremath{\varnothing}}

\newcommand{\cL}{\ensuremath{\mathcal L}}

\newcommand{\cD}{\ensuremath{\mathcal D}}

\newcommand{\cF}{\ensuremath{\mathcal F}}
\newcommand{\cP}{\ensuremath{\mathcal P}}

\newcommand{\sV}{\ensuremath{\mathscr V}}

\newcommand{\cC}{\ensuremath{\mathcal C}}

\newcommand{\sU}{\ensuremath{\mathscr U}}

\begin{document}

\openup 0.6em

\fontsize{13}{5}
\selectfont

	\begin{center}\LARGE Shore and Non-Block Points in Hausdorff \mbox{Continua}
	\end{center}
	
	\begin{align*}
	\text{\Large Daron Anderson }  \qquad \text{\Large Trinity College Dublin. Ireland }  
	\end{align*} 
	\begin{align*} \text{\Large andersd3@tcd.ie} \qquad \text{\Large Preprint February 2015}  
	\end{align*}$ $\\

	\begin{center}
		\textbf{ \large Abstract}
	\end{center}
	
	\noindent  We study the shore and non-block points of non-metric continua. 
	We reduce the problem of showing a continuum to have non-block points to that of showing an indecomposable continuum to have non-block points. 
	As a corollary we prove that separable continua have at least two non-block points -- and moreover are irreducible about their set of non-block points.

	\section{Introduction} 
	\noindent In 1923 Moore \cite{noncutmoore} proved that every metric continuum has two or more non-cut points. 
	Whyburn \cite{noncut} extended this result in 1968 to cover non-metric continua. 
	Shore points were introduced by Puga-Espinosa et al \cite{PugaPuga,Puga01,Puga03} 
	as a strengthening of the notion of a non-cut point  and used in their study of dendrites. 
	
	Recently Leonel \cite{Leonel01} has improved upon the result of Moore by showing every metric continuum has two or more shore points. 
	Bobok et al \cite{B} pointed out how the two points discovered by Leonel are not only shore points, 
	but satisfy a stronger property which they called being a non-block point. 
	
	As the metric assumption is not necessary to guarantee the existence of non-cut points, 
	this prompts the question of whether it is necessary to guarantee the existence of shore or non-block points. 
	
	We reduce the existence problem to the case of indecomposable continua. 
	Moreover we show that if each indecomposable member of a class of continua $\cP$ -- which is closed under certain quotient maps -- 
	has two or more non-block points then every member of $\cP$ has two or more non-block points, 
	and is moreover irreducible about its set of non-block points. 
	We apply this result to prove the existence of non-block points in separable continua. 
	\section{Notation and Terminology}
	
	\noindent For sets $A$ and $B$ define $A - B = \{ a \in A \colon a \notin B\}$. 
	If $B$ is the singleton $\{b\}$ we will write $A-b$ without confusion. 
	Whenever we write $A \subset B$ we do not presume $A$ is a proper subset of $B$.
	
	For a subset $S \subset X$ denote by $S^\circ$ and $\overline S$ the interior and closure of $S$ respectively. 
	The {\it boundary} of $S$ is the subset $\overline S \cap \overline {(X-S)}$.

	$X$ is a nondegenerate Hausdorff continuum. 
	That is to say a compact connected Hausdorff space that contains more than one point. 
	$C(X)$ denotes the hyperspace of subcontinua of $X$ with the Vietoris topology. 
	
	The subset $P \subset X$ is called a {\it cut set} if $X-P$ fails to be connected. 
	In this case $P$ is said to cut $X$. 
	If $\{p\}$ is a cut set, we say $p$ is a {\it cut point} and that $p$ cuts $X$. 
	If $p$ is not a cut point it is called a {\it non-cut point}.
	
	The subset $P \subset X$ is called a {\it shore set} if $X$ is the limit in $C(X)$ of a net of subcontinua of $X-P$. 
	Equivalently for each finite family of nonempty open sets $U_1, U_2, \ldots, U_n$ 
	some subcontinuum in $X-P$ meets each of $U_1, U_2, \ldots, U_n$. If $\{p\}$ is a shore set we say $p$ is a {\it shore point}.
	
	The subset $P \subset X$ is called a {\it non-block set} if there exists a family of subcontinua of $X-P$ 
	whose intersection is nonempty and whose union is dense in $X$. 
	If $\{p\}$ is a non-block set we say $p$ is a {\it non-block point}.
	
	We write $X = Y \oplus Z$ to mean that $Y$ and $Z$ are proper subcontinua of $X$ for which $X = Y \cup Z$. 
	Then $Y$ and $Z$ are said to form a {\it decomposition} of $X$. 
	In case $X$ admits no decomposition it is said to be {\it indecomposable}.
	The subcontinuum  $K \subset X$ is said to be {\it thick} if it is both proper and has nonvoid interior. 
	Being indecomposable is equivalent to having no thick subcontinua.
	
	For each $x \in X$ denote by $\kappa(x)$ the {\it composant} of $x$, meaning the union of all proper subcontinua that contain $x$.
	$X$ is called {\it aposyndetic at} $ x$ {\it with respect to} $y$ if $x$ is contained in the interior of a subcontinuum disjoint from $y$. 
	$X$ is called {\it aposyndetic at} $x$ if it is aposyndetic at $x$ with respect to each $y \ne x$. 
	$X$ is said to be {\it aposyndetic} if it is aposyndetic at each point. 
	$X$ is said to be {\it null-aposyndetic} at $x$ if it is aposyndetic at $x$ with respect to no other point of $X$. 
	In other words no proper subcontinuum contains $x$ in its interior. 
	In this case we also refer to the point $x$ as null-aposyndetic.
	
	The subset $A \subset B$ of a partially ordered set $B$ is said to be {\it cofinal} if each element of $B$ is bounded from above by an element of $A$. 
	Define the {\it cofinality} of a totally-ordered set $B$ as the unique least ordinal $\B$ which is order-isomorphic to a cofinal subset of $B$.

	\section{Coastal Continua}

	\noindent It was proved by Leonel \cite{Leonel01} that every metric continuum has two or more shore points. This improves the classical result of Moore \cite{noncutmoore} that every metric continuum has two or more non-cut points. Bobok et al \cite{B} observed how the two shore points discovered by Leonel are in fact non-block points. (They also show that a shore point of a metric continuum need not be a non-block point.) It is unknown whether these results extend to Hausdorff continua.
	
	\begin{definition}

		For a subcontinuum $K \subset X$ and subset $P \subset X$ define the {\it composant} of $ K$ {\it relative to} $P$:
		
		\begin{center}$ \displaystyle \kappa(K; P) = \bigcup \big \{M \in C(X) \colon K \subset M, M \ne X,M \cap P = \0 \big \}$ \end{center}
		
		When for example $K = \{x\}$ and $P = \{p\}$ we will write $\K(x;p)$. Note that if $P = \0$ then $\K(x;P)$ is the composant of $x$. That the composant of a point is dense follows from the boundary bumping theorem for Hausdorff continua. The proof is found in $\S$47, III Theorem 2 of \cite{kur2}.
	\end{definition}
	
	The following theorem has been proved by Bing \cite{Bing01} when $X$ is metric. The proof immediately generalizes to spaces whose degree of Baireness is no less than the weight.
	
	\begin{theorem} \label{Bing} For each $x \in X$ some $p\in X$ makes $\kappa(x;p)$ a dense subspace.\end{theorem}
	
	This motivates the following definition.
	
	\begin{definition} The continuum $X$ is called {\it coastal} at $x \in X$ to mean $\kappa(x;p)$ is dense for some $p\ne x$. 
		We call $X$ {\it coastal} to mean it is coastal at each point.
		
	\end{definition}
	
	By Theorem~\ref{Bing} all metric continua are coastal. Leonel \cite{Leonel01} essentially proved the following. 
	
	\begin{theorem} A coastal continuum has at least two non-block points.\end{theorem}
	
	Bobok, Pyrih and Vejnar \cite{B} have proved each metric continuum is irreducible about its set of non-block points. This strengthens the classic result that metric continua are irreducible about their sets of non-cut points. Theorem \ref{irr} shows how this result can be generalized. Since maps of the following type are ubiquitous later they warrant a name.
	
	\begin{definition} Let $K \subset X$  be a subcontinuum. The canonical quotient map $\theta \colon X \to X/K$  obtained by treating $K$ as a single point is called the $K${\it -bloom}. A map from $X$ is called a {\it bloom} if it is the $K$-bloom for some subcontinuum $K \subset X$. A class $\cP$ of continua is called a {\it bloom class} if it is closed under blooms.\end{definition}
	
	Note that blooms are continuous and monotone. Therefore the image of $X$ under the $K$-bloom is compact and connected. Moreover it follows from normality of $X$ that $X/K$ is Hausdorff, hence a continuum.
	\\
	
	Examples of bloom classes are the class of metric continua, the class of separable continua, and the class of finitely-irreducible continua. We will make frequent use of the following lemma.

	\begin{lemma} \label{pul} Let $K \subset X$ be a subcontinuum. 
		Suppose $X/K$ is coastal at the point $K$. Then $\K(K;y)$ is dense for some $y \notin K$. \end{lemma}
	
	\begin{proof} Suppose $X/K$ is coastal at the point $K$. 
		Then there exists a point $m \ne K$ of $X/K$ and  a family $\cC$ of subcontinua of $X/K$ 
		whose every element contains the point $K$ and is disjoint from $m$, such that $\bigcup \cC$ is dense in $X/K$.
		
		The point $m \in X/K$ is the image of a unique point $y \in X$ under the $K$-bloom $\theta$. 
		Since $\theta$ is monotone we have for each $C \in \cC$ that $\theta^{-1}(C)$ is a subcontinuum of $X$ containing $K$. 
		Moreover $y \notin \theta^{-1}(C)$ as $C$ is disjoint from $m$. 
		Therefore since $\theta$ is surjective $\theta^{-1} \big (\bigcup \cC \big) = \bigcup \big \{\theta^{-1}(C): C \in \cC \big \}$ 
		is a family of subcontinua whose union is dense in $X$ and disjoint from $y$.\end{proof}

	\begin{corollary} \label{pull} Let $K \subset X$ be a subcontinuum with $x \in K$. Suppose $X/K$ is coastal at the point $K$. Then $X$ is coastal at $x$.\end{corollary}
	
	\begin{theorem} \label{irr} Suppose the elements of the bloom class $\cP$ are coastal. Each element of $\cP$ is irreducible about its set of non-block points.\end{theorem}
	
	\begin{proof} Suppose $X$ is of class $\cP$. Denote by $B$ the nonempty set of non-block points of $X$ and assume a proper subcontinuum $K$ contains $B$. Choose a point $x \in K$. Since $\cP$ is a bloom class $X/K$ is coastal at the point $K$. 
		By Lemma \ref{pul} there is some $y \notin K$ for whick $\K(K;y)$ is dense in $X$. 
		But since $K \subset \K(K;y)$ this implies $y \notin K$ is a non-block point of $X$, contradicting the assumption that $B \subset K$.  \end{proof}
	
	Once we have shown in Section 1.5 that separable continua are coastal, 
	Theorem \ref{irr} can be employed to show separable continua are irreducible about their set of non-block points.

	\section{Blooms}
	
	\noindent The purpose of this section is to show that in order to prove all continua are coastal we may restrict our attention to indecomposable continua. 
	As an intermediate step we shall first reduce the problem of showing all continua to be coastal to that of showing them to be coastal at their null-aposyndetic points. We first prove some results about null-aposyndetic points for later use.  
	
	\begin{lemma}\label{thick}Each thick subcontinuum $T \subset X$ contains all null-aposyndetic points of $X$.\end{lemma} 
	\begin{proof} Suppose $x \in X$ is null-aposyndetic. Since $T$ is thick $\overline {(X-T)}$ is a proper subset of $X$. 
		Suppose first $\overline {(X-T)}$ is connected. Then we have a decomposition $X = $ \mbox{$T \oplus \overline {(X-T)}$}. 
		But $X-T$ is contained in the interior of the proper subcontinuum $\overline {(X-T)}$. Therefore $x \notin X-T$ and as a result $x \in T$.
		
		Now suppose $\overline{(X-T)}$ is disconnected.
		Then $X-T$ is disconnected and there exist nonempty disjoint clopen subsets $A$ and $B$ of $X-T$ 
		such that $A \cup B = X-T$. 
		If $x \notin T$ then without loss of generality $x\in A$. 
		By boundary bumping $T\cup A$ and $T \cup B$ are proper subcontinua 
		and $X = (T \cup A) \oplus (T \cup B)$ is a decomposition. 
		But this implies $x$ is an element of the open subset $X -(T\cup B)$ of the continuum $T \cup A$, contradicting how $x$ is null-aposyndetic.
	\end{proof}
	
	Only indecomposable continua fail to contain a thick subcontinuum. In this case all points are null-aposyndetic.

	\begin{lemma} \label{aab} Each thick subcontinuum $T \subset X$ contains all null-aposyndetic points of $X$ in its boundary.\end{lemma}
	
	\begin{proof} Suppose $x \in X$ is null-aposyndetic. 
		The previous lemma says $x \in T$. 
		But since $x$ is null-aposyndetic it is not in the interior of any proper subcontinuum.
		Therefore $x$ is an element of the boundary of $T$. \end{proof}
	
	\begin{corollary} \label{voidd} If $x$ is null-aposyndetic each subcontinuum that does not contain $x$ has void interior.\end{corollary}
	
	We are now ready to prove our first main theorem.

	\begin{theorem} \label{red} Suppose there exists a non-coastal continuum. There also exists a continuum that is not coastal at a null-aposyndetic point.\end{theorem}
	
	\begin{proof}
		
		Suppose $X$ fails to be coastal at $x$. We will construct a proper subcontinuum $M \subset X$. The image of $X$ under the $M$-bloom will fail to be coastal at a null-aposyndetic point.

		Let $(D^\A)_{\A<\xi}$ be an open base for the topology of $X$ where $\xi$ is a cardinal. Let $N_0 = \{ x\}$ and define a nest of proper subcontinua $N_\alpha$ by transfinite recursion as follows: 
		
		Where $\A = \B+ 1$ is a successor ordinal: If some proper subcontinuum $R$ contains $N_\B$ in its interior and meets $D^\A$ then let $N_\A = R$. If no such subcontinuum exists let $N_\A = N_\B$.
		
		Where $\A$ is a limit ordinal: Consider $L_\A = \bigcup _{\B < \A}N_\B$. We cannot have $L_\A = X$ since then the interiors of the sets $N_\B$ would form an increasing chain of proper open subsets with union $X$, and compactness of $X$ forbids this. $L_\A$ cannot be dense as then $\K(x;p)$ would be dense for every $p \notin L_\A$ which would imply $X$ to be coastal at $x$. 
		If some proper subcontinuum $R$ contains $\overline {L_\A}$ in its interior and meets $D^\A$ then let $N_\A = R$. 
		If no such subcontinuum exists let $N_\A =  \overline {L_\A}$.

		The union $N = \bigcup _{\B < \xi}N_\B$ must be proper and non-dense for the same reason each $L_\A$ is proper and non-dense.
		Define the proper subcontinuum $M =  \overline {N}$. Let $\pi\colon X \to X/M$ be the $M$-bloom and $X' = X/M$. Note that $\pi(M) = \pi (x)$ as $x \in M$. It follows from Corollary \ref{pull} that $X'$ fails to be coastal at $\pi(x)$.

		We claim that $X'$ is null-aposyndetic at $\pi(x)$. To prove this suppose there is an open set $U \subset X'$ and proper subcontinuum $K \subset X'$ such that $\pi(x) \in U \subset K$. Since $\pi$ is continuous and monotone we get an open set $\pi^{-1}(U)  \subset X$ and proper subcontinuum $\pi^{-1}(K) \subset X$ such that $M \subset \pi^{-1}(U) \subset \pi^{-1}(K)$. Then $\pi^{-1}(K) - M$ contains a basic open set $D^\A$ disjoint from $M$. Therefore $N_\A$ is disjoint from $D^\A$. Consider what happened at stage $\A$ of our construction. 
		
		Where $\A = \B +1$ is a successor ordinal:  It follows that $N_\A = N_\B$. But this implies no proper subcontinuum contains $N_\B$ in its interior while intersecting $D^\A$. But $\pi^{-1}(K)$ is a proper subcontinuum containing $N_\B$ in its interior, leading to a contradiction.
		
		Where $\A$ is a limit ordinal: It follows that no proper subcontinuum $R$ contains $\overline {L_\A }$ in its interior while intersecting $D^\A$. 
		But $\pi^{-1}(K)$  is a proper subcontinuum containing $\overline {L_\A }$ in its interior, leading to a contradiction.
	\end{proof}
	
	The previous theorem demonstrates how in proving all continua are coastal it is sufficient to look at the null-aposyndetic points. We now proceed to further reduce the problem to examining indecomposable continua.
	
	For the remainder of the section $X$ is a fixed continuum and $x \in X$ a fixed null-aposyndetic point. For $X$ indecomposable we are already done. Hence for convenience we assume $X$ is decomposable and so contains at least one thick subcontinuum. 
	
	
	\begin{lemma} The set of null-aposyndetic points of $X$ is nowhere dense.\end{lemma}
	\begin{proof} Let $T$ be a thick subcontinuum. By Lemma \ref{aab} we know each null-aposyndetic point is contained in the boundary of $T$. But the boundary of $T$ is closed with void interior. Therefore the set of null-aposyndetic points is nowhere dense.  \end{proof}
	
	The next lemma is used in the proof of Theorem \ref{else} which is itself a precursor to Theorem \ref{''}.

	\begin{lemma} \label{family}For each thick subcontinuum $T \subset X$ the point $x$ is in the closure of each component of $X-T$.\end{lemma}
	
	\begin{proof} If $X-T$ is connected $\overline {(X-T)}$ is a thick subcontinuum. 
		Then it contains $x$ by Lemma \ref{thick}. 
		Otherwise we can assume the family $\mathcal L$ of components of $X-T$ has more than one element. Also assume for a contradiction some $L \in \mathcal L$ does not have $x$ in its closure. 
		Corollary \ref{voidd} says $\overline {L}$ and hence $L$ have void interior. Fix some $p \in L$.
		
		By boundary bumping $T\cup R$ is a subcontinuum for each $R \in \cL$. Define the family $\cF = \{T \cup R \colon R \in \mathcal L, R \ne L\}$ of subcontinua. Lemma \ref{thick} says $x \in T$. Hence we have $\K(x,p)  \subset \bigcup \cF = X-L$.
		Since $L$ is nowhere dense this shows $\K(x,p)$ is dense. Thus $x$ is coastal contrary to assumption. 
		We conclude that $x \in \overline {L}$ for each component $L$ of $X-T$. \end{proof}
	
	\begin{theorem} \label{else} For each thick subcontinuum $T \subset X$ the set $\overline {(X-T)}$ is a thick subcontinuum.\end{theorem}
	
	\begin{proof} By Lemma \ref{family} every component of $\overline {(X-T)}$ meeting $X-T$ must contain $x$. 
		Therefore only one component $L$ of $\overline {(X-T)}$ meets $X-T$ and all other components are contained in the boundary of $T$. 
		But this implies that \mbox{$(X-T) \subset L$}. Since $L$ is connected it follows that $\overline {L} = \overline {(X-T)}$ is also connected.\end{proof}
	
	The next lemma readily follows from $-$ and will be employed in the proof of $-$ the stronger statement of Corollary \ref{interior}: that every thick subcontinuum has connected interior.

	\begin{lemma} \label{Comps} Let $T \subset X$ be a subcontinuum. The components of $T^\circ$ have nonvoid interior.\end{lemma}
	\begin{proof}This is clear if $T^\circ = X$ or $T^\circ = \0$. Hence we can assume $T$ is thick.
		Note the relation $X- \overline {(X-T)} = T^\circ$. Therefore the components of $X - \overline {(X-T)}$ are the components of $T^\circ$. 
		
		Let $\mathcal L$ denote the family of components of $T^\circ$. By Theorem \ref{else} $\overline {(X-T)}$ is a subcontinuum. 
		It follows from boundary bumping that $L \cup \overline {(X-T)}$ is a subcontinuum for each $L \in \cL$. 
		If some $L \in \cL$ has void interior then the family $\cF = \{R\cup \overline {(X-T)}  \colon R \in \mathcal L, R \ne L\}$ can be used to repeat the argument of Lemma \ref{family} to show $X$ is coastal at $x$, contrary to our assumptions.
	\end{proof}
	
	Since $X$ is decomposable it has at least one thick subcontinuum $T$. Theorem \ref{else} says $\overline {(X-T)}$ is a second thick subcontinuum.
	It is certainly not the case that $X$ only has two thick subcontinua. For Theorem 5.5 of \cite{nadlerbook} shows how to expand a subcontinuum within a prescribed open set. Obviously expansion preserves the property of being thick. 
	
	However we will show $T$ and $\overline{(X-T)}$ comprise {\it almost all} thick subcontinua of $X$.
	That is to say we cannot expand either to a proper subcontinuum with a strictly larger interior.
	
	\begin{lemma} Suppose $U$ is an open subset of the subcontinuum $T \subset X$ and $x \notin \overline {U}$. Then $X-U$ is disconnected. 
		One component of $X-U$ contains $\overline {(X-T)}$ and all other components are contained in $T^\circ$.\end{lemma}
	\begin{proof} $X-U$ cannot be connected as then $\overline {(X-U)}$ would be a thick subcontinuum with boundary contained in $\overline {U}$. 
		Since $x \notin \overline {U}$ this contradicts Lemma \ref{aab}.
		
		Since $\overline {(X-T)}$ is connected by Theorem \ref{else} and disjoint from $U$ it is contained in a component of $X-U$. 
		Because $X - \overline {(X-T)} = T^\circ$ each remaining component of $X-U$ is contained in $T^\circ$.
	\end{proof}

	\begin{lemma} \label{Inside} Suppose $S,T \subset X$ are thick subcontinua with $T \subset S$. Then $S^\circ = T^\circ$.\end{lemma}
	
	\begin{proof} Suppose otherwise that $T^\circ \ne S^\circ.$ It follows that $S^\circ - T$ is nonempty and open. Then there exists $p \in S^\circ - T$. 
		Choose an open $U\subset T^\circ$ such that $x \notin \overline {U}$. Note also that $U \subset S^\circ$. Let $\mathcal L$ denote the family of components of $X-U$. 
		By the previous lemma $\cL$ has more than one element; some $N \in \cL$ contains $\overline {(X-S)}$; and all $R \ne N$ are contained in $S^\circ$.
		
		Each $R \in \cL$ meets the boundary of $U$. Since $U \subset T$ each $R$ meets $T$ and hence $T \cup R$ is a subcontinuum. Let $L \in \cL$ be such that $p \in L$. Note that $L \not \subset T$. 
		The closure $\overline {L}$ is formed by adding to $L$ a subset of $\overline {U}$. 
		Since $x \notin \overline {U}$ it follows from lemma \ref{voidd} that $\overline {L}$ has void interior. Hence $L$ has void interior.
		
		The family of subcontinua $\{ N \cup T \cup R \colon R \ne L\}$ contains $X-L$ in the union and shows $\K (x;p)$ is dense, contrary to our assumptions.
	\end{proof}
	
	\begin{theorem} \label{together} Suppose $S$ and $T$ are thick subcontinua. Either $S\cup T = X$ or $S^\circ = T^\circ$.\end{theorem}
	
	\begin{proof} If $S^\circ \ne T^\circ$ then without loss of generality $S^\circ - T$ is nonempty and open. $S \cup T$ is a subcontinuum since $T$ and $S$ meet at $x$ by Lemma \ref{aab}.  Then $T \subset S \cup T$ but the interior of $S\cup T$ is strictly larger than the interior of $T$. By the Lemma \ref{Inside} we know $S \cup T$ cannot be thick. Therefore $S \cup T = X$. \end{proof}
	
	As mentioned earlier we can combine Theorem \ref{together} with Lemma \ref{Comps} to get the following.
	
	\begin{corollary} \label{interior} Every thick subcontinuum has connected interior.\end{corollary}
	
	We are now ready to prove the result mentioned in the section's preamble.
	
	\begin{theorem} \label{split} $X = S\oplus T$ where $S$ and $T$ are indecomposable and $S \cap T$ is nowhere dense.  \end{theorem}
	
	\begin{proof} Let $T$ be a thick subcontinuum and $S = \overline {(X-T)}$. By Corollary \ref{interior} and Theorem \ref{together} we can replace each of $S$ and $T$ with the closure of the interior. 
		Hence we can assume $T = \overline {(T^\circ)}$ and $S = \overline {(S^\circ)}$ and still have $X = S \cup T$. Since $S^\circ$ and $T^\circ$ are disjoint the intersection $S \cap T$ is nowhere dense. We claim each of $S$ and $T$ is indecomposable under the subspace topology. 
		
		To prove this suppose some subcontinuum $K \subset T$ has interior in $T$. Then there exists an open set $U$ of $X$ such that $T \cap U \subset K$. 
		Since $T = \overline {(T^\circ)}$ the open set $U$ must intersect $T^\circ$. Then $U \cap T^\circ \subset K$ and therefore $K$ is thick in $X$.
		
		It follows from Lemma \ref{Inside} that $K^\circ = T^\circ$. 
		But since we assumed $K \subset T$ and $T = \overline {(T^\circ)}$ this guarantees $K = T$. 
		We conclude that no proper subcontinuum of $T$ has interior in $T$. 
		Equivalently $T$ is indecomposable under the subspace topology. 
		The proof is identical for $S$.
	\end{proof}
	
	The next lemma will allow us to obtain an indecomposable quotient by treating either of the two indecomposable continua $S$ or $T$ as a single point.
	
	\begin{lemma} \label{phi}Suppose $X = S \oplus T$ where $S$ and $T$ are indecomposable. The image of $X$ under the $S$-bloom is indecomposable.\end{lemma}
	\begin{proof}
		Since $T$ contains the open set $X-S$ it is thick. Likewise $S$ is thick. Therefore $x \in S \cap T$. Let $\varphi \colon X \to X/S$ be the $S$-bloom. We claim $X/S$ is indecomposable. To prove this suppose $K \subset X/S$ is a thick subcontinuum. Since blooms are monotone $\varphi^{-1}(K)$ is a thick subcontinuum of $X$. By Lemma \ref{thick} the null-aposyndetic point $x \in \varphi^{-1}(K)$. It follows that $\varphi(x) \in K$.

		Since $K$ has interior in $X/S$ there exists an open set $U$ of $X/S$ such that $\varphi(x) \notin U \subset K$. Then $x \notin \varphi^{-1}(U) \subset \varphi^{-1}(K)$. In particular $\varphi^{-1}(U)$ is disjoint from $S$. It follows that the interior of $\varphi^{-1}(K)$ is strictly larger than the interior of $S$. But then $\varphi^{-1}(K)$ is a subcontinuum of $X$ that contains $S$ and whose interior strictly contains the interior of $S$. Then theorem \ref{together} implies $\varphi^{-1}(K) = X$. This in turn implies $K = X/S$ contradicting the assumption that $K$ is thick. \end{proof}
	
	\begin{theorem} \label{''} If there exists a non-coastal continuum, there exists a non-coastal indecomposable continuum.\end{theorem}
	\begin{proof} Suppose $X$ fails to be coastal at $x$. Let $\pi \colon X \to X/M$ be the bloom constructed in Theorem \ref{red}. If $X/M$  is indecomposable let $\varphi$ be the identity mapping on $X/M$ and $S = \{M\}$.
		
		Otherwise $X/M$ is decomposable and fails to be coastal at the null-aposyndetic point $\pi(x)$. In this case write $X/M = S \oplus T$ as the union of two indecomposable subcontinua and let $\varphi$ be the $S$-bloom constructed in Lemma \ref{phi}. 
		
		In both cases the continuum $\varphi \big (\pi (X)\big ) = (X/M)/S$ is indecomposable. Let $\theta =  \varphi \circ \pi$. It follows that $\theta$ is the $K$-bloom for $K = (\pi^{-1}(S) \cup M)$. Corollary \ref{pull} then implies that $\theta(X)$ fails to be coastal at $\theta(x)$. 
	\end{proof}
	
	\begin{corollary} \label{nix} If all indecomposable continua are coastal, all continua are coastal.\end{corollary}
	
	Corollary \ref{nix} can be strengthened by letting $X$ be an element of some bloom class.
	
	\begin{theorem} \label{reduce} Suppose $\mathcal P$ is a bloom class. If all indecomposable continua of type $\mathcal P$ are coastal, all continua of type $\mathcal P$ are coastal. \end{theorem}
	
	\begin{proof} Let $X$ be a member of the bloom class $\cP$ and $x \in X$. Let the blooms $\pi\colon X \to X/M$ and $\varphi \colon (X/M) \to (X/M)/S$ be as defined in the proofs of Theorems \ref{red} and \ref{phi} respectively. It follows that $\theta =  \varphi \circ \pi$ is the $K$-bloom for $K = (\pi^{-1}(S) \cup M)$.
		
		From the proof of Theorem \ref{''} we know $\theta(X)$ is nondegenerate and indecomposable. Since $\cP$ is a bloom class $\theta(X) \in \cP$. By assumption $\theta(X)$ is coastal at $\theta(x)$. Then by Corollary \ref{pull} we have that $X$ is coastal at $x$.
	\end{proof}

	Theorem \ref{reduce} provides an alternate proof of Bing's Theorem \ref{Bing} in the metric case: Let $\mathcal P$ denote the class of metric continua. 
	Observe that if a metric space is mapped continuously onto a Hausdorff space, the image is metrizable. From this we conclude that $\cP$ is closed under blooms. Therefore by Theorem \ref{reduce} we may restrict our attention to indecomposable metric continua. Recall an indecomposable metric continuum has uncountably many pairwise disjoint composants. In that case we may, for each $x \in X$, choose $p \in X - \K(x)$. Then $\K (x;p) = \K(x)$ is dense. 
	
	We observe that if $X$ is not coastal at $x$ then $\theta(X)$ is the composant of $\theta(x)$. We can say slightly more.
	
	\begin{lemma} Suppose $X$ fails to be coastal at $x$. There exists a subcontinuum $K \subset X$ including $x$ such that $X/K$ is indecomposable and fails to be coastal at the point $K$. Therefore $X/K$ has exactly one composant.\end{lemma}
	
	\begin{proof}
		
		In the notation of Theorem \ref{''} take $K = (\pi^{-1}(S) \cup M)$ and let $\theta$ be the $K$-bloom. Then $\theta(X)$ is indecomposable. Moreover Corollary \ref{pull} implies that $\theta(X)$ is not coastal at $\theta(x)$. Therefore $\theta(X)$ is the composant of $\theta(x)$; 
		otherwise $\K \big (\theta(x);p \big )$ would be dense for each $p \notin \K \big (\theta(x) \big )$. It is well known that distinct composants of an indecomposable continuum are disjoint. Therefore $\theta(X)$ must be the composant of each of its points. Therefore there is exactly one composant.
	\end{proof}
	
	Thus the study of which continua are coastal reduces to the study of indecomposable continua with exactly one composant. Continua of this sort are a peculiarity of the non-metric realm and were shown to exist by Bellamy \cite{one}.

	\section{Baireness}
	
	\noindent The non-block point existence theorem of Leonel relies on Bing's Theorem \ref{Bing}. A slightly modified version of his proof applies to continua that satisfy a condition on its cardinal invariants. We first define the relevant cardinal invariants, then give the condition itself.
	
	\begin{definition} A space is called $\A$-Baire if every family of $\A$ many open dense subsets has dense intersection.\end{definition}
	
	\begin{definition} The {\it weight} $w(X)$ of the space $X$ is the least cardinality of an open base for the topology.\end{definition}
	\begin{definition} The {\it density} $d(X)$ of the space $X$ is the least cardinality of a dense subset.\end{definition}
	
	Bing's proof of Theorem \ref{Bing} can be adapted to show that if the continuum $X$ is $w(X)$-Baire, it is coastal. If $X$ is metric it is second-countable and thus $w(X) = \aleph _0$. In addition metric continua satisfy the Baire category theorem and hence are $\aleph_0$-Baire. Combining these two facts yields that each metric continuum $X$ is $w(X)$-Baire.
	
	In this section we show it is enough to demand $X$ be $d(X)$-Baire. Denote by $\cD$ the family of continua meeting this condition and recall that $\cD$ contains all separable continua. Since it follows immediately from their definitions that $d(X) \le w(X)$ our result is a direct strengthening of Bing's.
	
	We will require the following well-known facts about cardinal invariants.
	
	\begin{proposition} \label{Baire} Suppose $X$ is $\A$-Baire and $U \subset X$ open. Then $U$ is $\A$-Baire. \end{proposition}

	\begin{proposition} \label{Dprop} Suppose $U \subset X$ is open. Then $d(U) \le d(X)$.\end{proposition}
	
	Let the blooms $\pi$ and $\varphi$ be as defined in the proofs of Theorems \ref{red} and \ref{phi} respectively. In the proof of Theorem \ref{reduce} we showed that $\theta  = \varphi \circ \pi$ is the  $K$-bloom for some subcontinuum $K \subset X$. 
	It follows from the surjectivity of $\theta$ that $d\big (\theta(X)\big ) \le d(X)$. 
	
	We can invoke Proposition \ref{Baire} for $U = X-K$ to get the following fact: If $X$ is $\A$-Baire then $\theta(X)$ is $\A$-Baire. In other words $\theta$ cannot reduce the degree of Baireness. Combining these observations gives the following.
	
	\begin{corollary} \label{bloom} The family of continua $\cD$ is a bloom class. \end{corollary}
	
	The main result of this section relies on the following lemma.
	
	\begin{lemma} \label{nest} Suppose $X \in \cD$. Each nest of open dense subsets  of $X$ has dense intersection.\end{lemma}

	\begin{proof} Suppose $\sU$ is a nest of open dense subsets of $X$.
		Since the $\sU$ is totally-ordered by reverse-inclusion there is a well-ordered cofinal subset $\mathscr V \subset \sU$. By cofinality we have $\bigcap \sV = \bigcap \sU$. By replacing $\sU$ with $\sV$ we can assume $\sU$ has the form $\{U(\G): \G < \Gamma\}$ for some ordinal $\Gamma$.
		Without loss of generality $\Gamma$ equals its own cofinality.
		
		Now suppose $V \subset X$ is an arbitrary open set. Choose a dense $D \subset V$ with $|D| = d(V)$.
		Without loss of generality $d(V)$ is an ordinal.
		Choose a dense subset $\big \{x(\A): \A < d(V) \big \}$ of $V$. 
		For a contradiction assume $V$ is disjoint from $\bigcap \sU$.
		In particular $x \notin \bigcap \sU$ for each $x \in D$. So we can define the function $\gamma:D \to \GG$ by $\G(x) = \min\big \{\A \in \GG: x \notin U(\A) \big \}$.
		
		Since each $U(\A)$ is open and dense there is some $x \in U(\A) \cap D$. It follows from definition $\A < \G(x)$. Letting $\A \in \GG$ be arbitrary we see $\G(D) \subset \GG$ is cofinal. Since $\Gamma$ equals its own cofinality $\Gamma \cong \G  (D)$ as ordinals.
		In particular $|\Gamma| = | \G (D )| \le |D| =  d(V)$ and so $|\GG| \le d(V)$.
		Proposition \ref{Baire} then implies $|\Gamma| \le d(X)$.
	\end{proof}

	\begin{corollary} \label{nest'} Suppose $X \in \cD$. The union of a nest of closed nowhere dense subsets of $X$ is proper.\end{corollary}
	
	The hypothesis $X \in \cD$ is essentially a restriction on the density of $X$.
	One might hope this hypothesis can be dropped. The next example shows otherwise.
	
	\begin{example} We give a continuum that is the union of a nest of nowhere dense subcontinua. The functional analysis terminology used in this example can be found in \cite{Rudin}. 
		
		Let $\W_1$ be the first uncountable ordinal and $E = \ell^2(\W_1)$ the Hilbert space of square-summable functions $x\colon \W_1 \to \RR$ with the inner-product $  xy =  $ $  \displaystyle \sum_{\A < \W_1} x(\A) y(\A) $. Define the $\A$-coordinate vector $e_\A$ by $e_\A(\B) = \delta_\A ^\B$ and define the $\B$-coordinate functional $\epsilon_\B$ by $\epsilon_\B(x) = x(\B)$.
		
		Our example is the closed unit ball $B \subset E$ under the weak topology. Theorem 3.15 of \cite{Rudin} says $B$ is weak$^{*}$ compact. But Theorem 12.5 of \cite{Rudin} implies a real Hilbert space is linear-isomorphic to its dual. Therefore $E = E^{**}$, the weak and weak$^*$ topologies coincide on $E$, and the closed unit ball $B \subset E$ is weak-compact. We represent $B$ as the union of a nest of nowhere-dense subsets. 
		
		For each $x \in B$ since $\| x \| < \infty$ only countably many values of $x(\A)$ are nonzero. Since $\W_1$ has uncountable cofinality some $\A < \W_1$ has $x(\B) = 0$ for all $\B > \A$. Let $B(\A) = \{ y \in B \colon y(\B) = 0 \ \forall \B > \A \}$. Then clearly \mbox{$x \in B(\A)$} and $ \bigcup _{\A < \W_1} B(\A) = B$. 
		Recall the functional $\epsilon_\B$ is by definition weak-continuous. Therefore $B(\A) = $ \mbox{$ B \cap \Big ( \bigcap _{\B > \A} \epsilon_\B ^{-1} (0) \Big )$} is closed in $B$. 
		
		To show $B(\A)$ is nowhere dense in $B$ assume otherwise $x \in B(\A)$ is an interior point. Write $\B=\A+1$ and recall $\epsilon_\B y = 0$ for all $y \in B(\A)$.
		Since addition and rescaling are weak-continuous so is the function $f(t) = x + te_{\B}$. 
		Observe $f(0) =x$ and since $x \in B(\A)$ is an interior point continuity gives $f(\E) \in B(\A)$ for some $\E >0$.
		Thus $x + \E e_{\B} \in B(\A)$. Now apply $\epsilon_\B$ to both sides to get $\epsilon_\B(x + \E e_{\B} ) = \epsilon_\B x + \E \epsilon_\B e_{\B}  = 0 + \E = 0$ which is a contradiction. We conclude $B(\A)$ has void interior relative to $B$.\end{example}

	\begin{theorem} \label{refindep} The indecomposable elements of $\cD$ are coastal.\end{theorem}
	
	\begin{proof} 
		Suppose $X$ is an indecomposable continuum of class $\cD$. Let $x  \in X$. Since $X$ is indecomposable distinct composants are disjoint. If $\K(x) \ne X$ choose $p \notin \K(x)$. Then $\K(x;p) = \K(x)$ is dense. 
		
		Now assume that $\K (x) = X$. Then $x$ can be joined to any point by a proper subcontinuum which must be nowhere dense since $X$ is indecomposable.
		
		Let $Y = \{y_\A \colon \A < d(X)\}$ be a dense subset with $x = y_0$. For each \mbox{$1 \le \A < d(X)$} let $L^\A$ be a subcontinuum with void interior joining $y_0$ to $y_\A$. We will define an increasing chain of proper subcontinua $N_\A$ by transfinite recursion on $\B$ beginning with $N_0 = \{y_0\}$.
		
		When $\G = \A + 1$ is a successor ordinal: Let $N_\G = N_\A \cup L^\G$. By construction each of $N_\A$ and $ L^\G$ is closed and nowhere dense, hence proper. It follows that $N_\G$ is closed, nowhere dense and proper as well. $N_\G$ is connected since each of $N_\A$ and $L^\G$ contains $x$.

		When $\G$ is a limit ordinal: Consider $M^\G = \bigcup _{\A < \G} N_\A$. By construction each $N_\A$ is closed and nowhere dense so Corollary \ref{nest'} implies $M^\G \ne X$. $M^\G$ is connected since each $N_\A$ contains $x$. If $M^\G$ is dense choose $p \notin M^\G$. Then $\K (x;p)$ contains each subcontinuum $N_\A$ so is dense. This proves $X$ is coastal at $x = y_0$. 
		Otherwise $\overline {(M^\G)}$ is a proper subcontinuum, so must be nowhere dense. In this case it follows that $N_\G = \overline {(M^\G)} \cup L^\G$ is a proper and nowhere dense subcontinuum.
		
		Assuming no $M^\G$ is dense consider the union $M = \bigcup _{\A < d(X)} N_\A$. This is proper by Corollary \ref{nest'}. Moreover it contains the dense subset $Y$. Therefore it is dense. It follows that $\K(x;p)$ is dense for each $p \notin M$. Thus $X$ is coastal at $x=y_0$.
	\end{proof}
	
	To prove the next theorem we apply  Theorem \ref{reduce} and Corollary \ref{bloom} to the class $\cD$. 
	
	\begin{corollary} \label{dX} Continua of class $\cD$ are coastal.\end{corollary}
	
	\begin{corollary}For each separable continuum $X$ and point $x \in X$ some $p \in X$ makes the continuum component of $x$ in $X-p$ dense.\end{corollary}
	
	Finally we can apply Theorem \ref{irr} to Corollary \ref{dX} to get the following.
	
	\begin{theorem} Every continuum whose density does not exceed its Baire characteristic $-$ in particular every separable continuum $-$ is irreducible about its set of non-block points.\end{theorem}
	
	\section*{Acknowledgements}
	
	The author would like to thank Professor Paul Bankston and Doctor Aisling McCluskey for their help in preparing the manuscript. We are also grateful to the referee for their attention and suggestions.











\end{document}